\documentclass[11pt, leqno]{amsart}
\usepackage{amssymb,amsthm}
\usepackage{geometry}               % See geometry.pdf to learn the layout options. There are lots.
\usepackage{graphicx}
\usepackage{nicefrac}
\usepackage{amsmath}
\usepackage{amsfonts}
\usepackage{amsxtra}     % Use various AMS packages
\usepackage{epsfig}
\usepackage{verbatim}

\DeclareMathOperator{\Vol}{Vol}
\DeclareMathOperator{\Ker}{Ker}
\DeclareMathOperator{\ad}{ad}
\DeclareMathOperator{\spn}{span}
\DeclareMathOperator{\tr}{trace}

\DeclareMathOperator{\SU}{SU}
\DeclareMathOperator{\SO}{SO}
\DeclareMathOperator{\GL}{GL}

\providecommand{\norm}[1]{\lVert#1\rVert}

 \newcommand{\be}{\begin{equation}}
       \newcommand{\ee}{\end{equation}}
       \newcommand{\ba}{\begin{eqnarray}}
        \newcommand{\ea}{\end{eqnarray}}
 \newcommand{\ban}{\begin{eqnarray*}}
 \newcommand{\ean}{\end{eqnarray*}}

 \newcommand{\lp}{\langle}
 \newcommand{\rp}{\rangle}
 \newcommand{\ra}{\rightarrow}

 \newcommand{\sect}[1]{\section{#1} \setcounter{equation}{0}}

 \newcommand{\vol}{\mathrm{Vol}}

\newtheorem{theorem}{Theorem}[section]
\newtheorem{prop}[theorem]{Proposition}
\newtheorem{corollary}[theorem]{Corollary}
\newtheorem{lemma}[theorem]{Lemma}

\newtheorem{definition}[theorem]{Definition}

\newtheorem{remark}[theorem]{Remark}

\title{On volumes of complex hyperbolic orbifolds}
\author{Ilesanmi Adeboye}
\address{Math and Computer Science Department\\
         Wesleyan University\\
         Middletown, CT 06459}
\email{iadeboye@wesleyan.edu}
\author{Guofang Wei}
\address{Department of Mathematics\\
         University of California\\
         Santa Barbara, CA 93106}
\email{wei@math.ucsb.edu}
\thanks{The second author is partially supported by NSF grant DMS-1105536}

\date{}
\begin{document}
\begin{abstract} We construct an explicit lower bound for the volume of a
complex hyperbolic orbifold that depends only on dimension.
\end{abstract}
\bibliographystyle{plain}
\maketitle
\thispagestyle{empty}
\section{Introduction}
A \textit{hyperbolic orbifold} is a quotient of real, complex, quaternionic or octonionic hyperbolic space, by a discrete group of isometries, usually denoted by $\Gamma$. An orbifold is a \textit{manifold} when $\Gamma$ contains no elements of finite order. 

The real hyperbolic 2-orbifold of minimum volume was identified by Siegal \cite{CLS}. An analogous result in dimension three was proved by Gehring and Martin \cite{GM1}. In the remaining dimensions, and algebras of definition, the existence of a hyperbolic orbifold of minimum volume is guaranteed by a theorem of Wang \cite{Wa1}.

In this paper, we prove an explicit lower bound for the volume of any complex hyperbolic orbifold that depends only on dimension. Our methods here are similar to those of the prequel \cite{Ade2}, which addressed the real hyperbolic case. The complex setting provides an additional corollary and the corresponding Lie group curvature calculations are of independent interest. 

Let $\mathbf H^n_{\mathbb C}$ denote complex hyperbolic $n$-space. The holomorphic sectional curvature of $\mathbf H^n_{\mathbb C}$ is normalized to be $-1$, accordingly, the sectional curvatures are pinched between $-1$ and $-1/4$. Let $\SU(n,1)$ denote the indefinite special unitary group of indicated signature. This group is a Lie group and it acts transitively by isometries on complex hyperbolic space. With an appropriate scale of a canonical metric on $\SU(n,1)$, we define a \textit{Riemannian submersion} $\pi: \SU(n,1)/\Gamma \rightarrow \mathbf H^n_{\mathbb C}/ \Gamma$. The volume of a complex hyperbolic $n$-orbifold is thereby described in terms of the volume of the fundamental domain of a lattice in a Lie group. The latter is then bounded from below using results due to Wang \cite{Wa1} and Gunther (see e.g. \cite{Gall}). In what follows dimension will refer to complex dimension, unless otherwise stated.  

\begin{theorem}\label{Bound} The volume of a complex hyperbolic $n$-orbifold is bounded below by $\mathcal C(n)$, an explicit constant depending only on dimension, given by
\[\mathcal C(n)=\frac{2^{n^2+n+1} \pi^{n/2}(n-1)!\,(n-2)!\cdots!\, 3! \,2! \,1!}{(36n+21)^{(n^2+2n)/2}\Gamma((n^2+2n)/2)}\int_{0}^{\min[0.06925\sqrt{36n+21},\pi]}\sin^{n^2+2n-1}\rho\,\,\,d\rho.\]
\end{theorem}

The formula of Theorem~\ref{Bound} gives a lower bound of 0.002 for complex 1-orbifolds and $2.918\times10^{-9}$ for complex hyperbolic 2-orbifolds. Since complex hyperbolic 1-space is isometric to real hyperbolic 2-space, a sharp volume bound of $\pi/21$ for complex hyperbolic 1-orbifolds follows immediately from the classical results of Hurwitz \cite{Hur} and Siegel \cite{CLS}.

As in \cite{Ade2}, we note that volume bounds for hyperbolic orbifolds provide immediate information on the order of the symmetry groups of hyperbolic manifolds. Following Hurwitz's formula for groups acting on surfaces, we have the following corollary.

\begin{corollary}\label{cor1} Let $M$ be an complex hyperbolic $n$-manifold. Let $H$ be a group of isometries of $M$. Then
\[
|H|\leq\frac{\Vol[M]}{\mathcal C(n)}.
\]
\end{corollary}

The Chern-Gauss-Bonnet formula (see e.g. \cite{HerPau}) describes volume in terms of Euler characteristic: \[\Vol(M)=\frac{(-4\pi)^n}{(n+1)!}\chi(M).\] This formula gives an explicit lower bound for the class of complex hyperbolic manifolds. In that case the Euler characteristic is integer valued. However, there is no such restriction for orbifolds. The formula also provides an alternate version of Corollary~\ref{cor1}.

\begin{corollary}\label{cor2} Let $M$ be a finite volume complex hyperbolic $n$-manifold. Let $H$ be a group of isometries of $M$. Then
\[
|H|\leq\frac{(-4\pi)^n}{\mathcal C(n)(n+1)!}\chi(M).
\]
\end{corollary}

The next section gives a definition of complex hyperbolic space and concludes with our preferred symmetric space representation. A comprehensive treatment of complex hyperbolic geometry can be found in \cite{Gold}. Section~\ref{sun} provides the background on the geometry of $\SU(n,1)$ that we will use, subsequently. For more details, the reader may consult Chapter 6 of \cite{Jost2} or Sections 1-3 of \cite{Ade2}, where we undertook a similar analysis of $\SO_o(n,1)$.

\sect{Complex Hyperbolic Space}\label{chs}

Let   $\mathbb C^{n,1}$ be a complex vector space of dimension $(n+1)$ equipped with the Hermitian form \[\langle \mathbf z,\mathbf w\rangle=\mathbf zJ\mathbf w^*=z_1\overline{w_1}+z_2\overline{w_2}+\cdots+z_n\overline{w_n}-z_{n+1}\overline{w_{n+1}},\] where \[J=\left(\begin{smallmatrix}
1 &   &   &   &   \\  & 1 &   &   &   \\  &   & \ddots &   &   \\  &   &   & 1 &   \\  &   &   &   & -1
\end{smallmatrix}\right)\] and $(\cdot)^*$ represents conjugate transpose. Note that, for all $\mathbf z\in\mathbb C^{n,1}$ and $\lambda\in\mathbb C$, $\langle \mathbf z,\mathbf z\rangle\in\mathbb R$ and $\langle \lambda\mathbf z,\lambda\mathbf z\rangle=|\lambda|^2\langle\mathbf z,\mathbf z\rangle$. 

Let \[V_-=\{\mathbf z\in\mathbb C^{n,1}: \langle \mathbf z,\mathbf z\rangle<0\},\] and let \[\mathbb P:\mathbb C^{n,1}-\{0\}\to\mathbb C\mathbf P^n\] be the canonical projection onto complex projective space. \textit{Complex hyperbolic n-space}, $\mathbf H^n_{\mathbb C}$, is defined to be the space $\mathbb P(V_-)$ together with the \textit{Bergman metric}. The Bergman metric is defined by the distance function $\rho$, given by the formula \[\cosh^2\left(\frac{\rho(z,w)}{2}\right)=\frac{\langle\mathbf z,\mathbf w\rangle\langle\mathbf w,\mathbf z\rangle}{\langle\mathbf z,\mathbf z\rangle\langle\mathbf w,\mathbf w\rangle},\] where $\mathbf z$ and $\mathbf w$ are lifts of $z,w\in\mathbf H^n_{\mathbb C}.$

Let $\GL(n,\mathbb{C})$ be the group of complex nonsingular $n$-by-$n$ matrices. The \textit{unitary group} is defined and denoted by \[U(n)=\{A\in \GL(n,\mathbb{C})| AA^*=I\}.\]
Denote by $U(n,1)$ the group of all linear transformation of $\mathbb C^{n,1}$ which leave the form $\langle \mathbf z,\mathbf w\rangle$ invariant. That is, \[U(n,1)=\{A\in \GL(n+1,\mathbb{C})| AJA^*=J\}.\] 

The induced action of $U(n,1)$ on $\mathbb C\mathbf P^n$ preserves $\mathbf H^n_{\mathbb C}$ and acts by isometries. The stabilizer of the point of $\mathbf H^n_{\mathbb C}$ with homogeneous coordinates $[0:\cdots:0:1]$ is \[U(n)\times U(1)=\left\{\left(\begin{array}{cc}A & 0 \\0 & e^{i\theta}\end{array}\right)\Big|A\in U(n), ~\theta\in[0,2\pi)\right\}.\]
We can identify $U(n)$ with the elements of $U(n)\times U(1)$ that have determinant $1$ by the map \[A\rightarrow \left(\begin{array}{cc}A & 0 \\0 & (\det A)^{-1}\end{array}\right).\]

Hence, \[\mathbf H^n_{\mathbb C}=U(n,1)/\left\{U(n)\times U(1)\right\}=\SU(n,1)/S\left\{U(n)\times U(1)\right\}=\SU(n,1)/U(n).\]

\sect{The Lie group $\SU(n,1)$}\label{sun}

A \textit{matrix Lie group} is a closed subgroup of $\GL(n,\mathbb{C})$. Recall that for a square matrix $X$, \[e^X=I+X+\frac12X^2+\cdots. \] The Lie algebra of a matrix Lie group $G$ is a vector space, defined as the set of matrices $X$ such that $e^{tX}\in G$, for all real numbers $t$. The Lie algebra of $\GL(n,\mathbb{C})$, denoted by $\mathfrak{gl}(n,\mathbb{C})$, is the set of all $n\times n$ matrices over $\mathbb C$.

The indefinite special unitary group,  \[\SU(n,1)=\{A\in U(n,1): \det A=1\},\] is a matrix Lie group of real dimension $n^2+2n$. The Lie algebra of $\SU(n,1)$ is defined and denoted by \[\mathfrak{su}(n,1)=\{X\in \mathfrak{gl}(n,\mathbb{C})|JX^*J=-X, \ \tr X =0 \}.\]

\begin{definition} \label{basiselt} For each $n$, let $e_{jk}\in\mathfrak{gl}(n+1,\mathbb{C})$ be the matrix with 1 in the $jk$-position and 0 elsewhere. Furthermore, let $\alpha_{jk}=(e_{jk}-e_{kj}), \beta_{jk}=(e_{jk}+e_{kj})$ and $h_j=i(e_{jj}-e_{n+1,n+1}).$ The \textbf{standard basis} for  $\mathfrak{su}(n,1)$, denoted by $\mathfrak{B}$, consists of the following set of  $n^2+2n$ matrices:
\[ \begin{array}{ll} \alpha_{jk}, \ 1\le j <k \le n ,  &  i \beta_{jk}, \ 1\le j < k \le n,  \\
\beta_{j,n+1}, \ 1 \le j \le n, & i \alpha_{j,n+1}, \  1 \le j \le n, \\
h_j, \ 1 \le j \le n.
\end{array}\]
\end{definition}

The Lie bracket of a matrix Lie algebra is determined by matrix operations:  \[ [X,Y] = XY-YX.\]
 The following proposition describes the Lie bracket of  $\mathfrak{su}(n,1)$. The proof involves straightforward calculation and is omitted.
 
 \begin{prop}  \label{multtable}
For $1\le j < k\le n, 1\le l < m\le n,$
\be
[\alpha_{jk}, \alpha_{lm}] = \delta_{kl} \alpha_{jm} + \delta_{km}\alpha_{lj} + \delta_{jm} \alpha_{kl}+ \delta_{lj}\alpha_{mk},
\ee
\be
[i\beta_{jk}, i\beta_{lm}] = -(\delta_{kl} \alpha_{jm} + \delta_{km}\alpha_{jl} + \delta_{jm} \alpha_{kl}+ \delta_{lj}\alpha_{km} ),
\ee
\be
[h_j, h_k] = 0,
\ee
\be
  [\alpha_{jk}, i\beta_{lm}] = i (\delta_{kl} \beta_{jm} + \delta_{km}\beta_{jl} - \delta_{jm} \beta_{kl}-\delta_{lj}\beta_{km} ) ,
\ee
\be
[\alpha_{jk}, h_l] = i(\delta_{kl}\beta_{jl} - \delta_{lj}\beta_{kl}),  \label{h-alpha-ij}
\ee
\be
[h_l, i\beta_{jk}] = \delta_{kl}\alpha_{jl} + \delta_{lj}  \alpha_{kl},   \label{h-beta-ij}
\ee

\be
[\alpha_{jk}, \beta_{l, n+1}] =  \delta_{lk}  \beta_{j,n+1} - \delta_{jl} \beta_{k, n+1},   \label{alpha-beta-n+1}
\ee
\be [\alpha_{jk}, i \alpha_{l,n+1}] = i(\delta_{kl}\alpha_{j,n+1} - \delta_{lj}\alpha_{k,n+1}),
\ee
\be
[ i\beta_{jk},  \beta_{l, n+1}] =  i(\delta_{lk}\alpha_{j,n+1} +\delta_{jl}\alpha_{k, n+1}),
\ee
\be [ i \beta_{jk}, i\alpha_{l, n+1}] = -( \delta_{lk} \beta_{j,n+1} + \delta_{jl}\beta_{k, n+1} ),
\ee
\be
[h_j, \beta_{l,n+1}] = i(\delta_{jl} \alpha_{j, n+1} +\alpha_{l,n+1}),  \label{h-beta-n+1}
\ee
\be
[h_j, i\alpha_{l,n+1}] = -(\delta_{jl}  \beta_{j, n+1}+\beta_{l,n+1}),  \label{h-alpha-n+1}
\ee

\be
[\beta_{j,n+1}, \beta_{k,n+1}]  =  \alpha_{jk}, \label{sigref}
\ee
\be
[ i\alpha_{j,n+1}, i \alpha_{k,n+1}]  =  \alpha_{jk},
\ee
\be
[i\alpha_{j,n+1}, \beta_{k,n+1}] = i (\beta_{jk} - 2 \delta_{jk} e_{n+1, n+1} ).  \label{n+1-n+1}
\ee
\end{prop}

\begin{remark}
Proposition~\ref{multtable} illustrates a \textit{Cartan decomposition} $\mathfrak{su}(n,1) = \mathfrak{k} \oplus \mathfrak{p}$, where
\be
\mathfrak{k} =\spn\{ \alpha_{jk}, i\beta_{jk}, h_j, 1\le j<k \le n\}, ~\mathfrak{p}= \spn\{  \beta_{j,n+1},i\alpha_{j,n+1}, 1 \le j\le n\},\label{CD}
\ee
\be
[\mathfrak k, \mathfrak k] \subset \mathfrak k, \  \ [\mathfrak k, \mathfrak p] \subset \mathfrak p, \ \mbox{and} \ [\mathfrak p, \mathfrak p] \subset \mathfrak k.   \label{brackets}
\ee
\end{remark}

\subsection{The Canonical Metric of $\SU(n,1)$}\label{canmet2}

For $X \in \mathfrak {su}(n,1)$, the \textit{adjoint action} of $X$ is the $\mathfrak {su}(n,1)$-endomorphism defined by the Lie bracket, \[\ad X(Y)=[X,Y].\]
The \textit{Killing form} on $\mathfrak {su}(n,1)$ is a symmetric bilinear form given by \[B(X,Y)=\tr(\ad X\ad Y).\]
A positive definite inner product on $\mathfrak {su}(n,1)$ is then defined by putting

\[ \langle X,Y\rangle = \left\{ \begin{array}{lll} B(X,Y) & \mbox{for } X,Y\in\mathfrak p, \\
-B(X,Y) & \mbox{for } X,Y\in\mathfrak k, \\ 0 & \mbox{otherwise. } \end{array} \right. \]

By identifying $\mathfrak {su}(n,1)$ with the tangent space at the identity of $\SU(n,1)$, we extend $\langle \cdot,\cdot\rangle$ to a left invariant Riemannian metric over $\SU(n,1)$. We denote this metric by $g$ and refer to it as the \textit{canonical metric} for $\SU(n,1)$.

\begin{lemma} For $X,Y\in\mathfrak B$,

\begin{displaymath}
\lp X,Y \rp =\left\{\begin{array}{ccr}
4n+4 &\rm{if     }  &X=Y \\
0& &\rm{otherwise.}
\end{array}
\right.
\end{displaymath}
\end{lemma}

\begin{proof} Let \begin{align*}
\mathfrak k_{jk} &=\spn\{ \alpha_{jk}, i\beta_{jk}\}, 1\le j < k \le n,\\
\mathfrak k_h & =\spn\{ h_j, j=1,\cdots ,n\},\\
\mathfrak p_j & =\spn\{ \beta_{j,n+1}, i\alpha_{j,n+1}\},1\le j \le n.
\end{align*}

By (\ref{h-alpha-ij}), (\ref{h-beta-ij}), (\ref{h-beta-n+1}) and (\ref{h-alpha-n+1}), for
each $h \in \mathfrak k_h$, $\ad h (\mathfrak k_{jk}) \subset \mathfrak k_{jk}$ and $\ad h (\mathfrak p_j)
\subset \mathfrak p_j$. In fact, if $h = \sum_s d_s h_s$, then
\[[h, \alpha_{jk} ] = (d_j - d_k) i\beta_{jk},  \]
\[ [ h,i\beta_{jk} ]  = ( d_k  - d_j) \alpha_{jk}, \]
\[ [h,\beta_{j,n+1}] = \left(d_j + \sum_s d_s\right) i\alpha_{j,n+1}, \]
\[ [h, i\alpha_{j,n+1}] = -\left(d_j + \sum_s d_s\right) \beta_{j,n+1}.  \]
  Therefore,
\ban B(h,h) = \tr \left((\ad h)^2\right) & = &  \sum_{j<k}  \tr ((\ad h|_{\mathfrak k_{jk}})^2) + \sum_j  \tr ((\ad h |_{\mathfrak p_j})^2) \nonumber \\
& = & - \sum_{j,k} (d_j  - d_k )^2 -2 \sum_j \left(d_j +\sum_k d_k\right)^2 \nonumber\\
& = & -(2n+2) \left( \sum_j d_j^2 +(\sum_j d_j)^2 \right)  \nonumber  \\
& = & (2n+2) (\tr h)^2.  \ean Since  each element of
$\mathfrak{su} (n,1)$ can be  diagonalized, for each  $X \in
\mathfrak{su} (n,1)$, there is a matrix $A$ such that $AXA^{-1}
\in \mathfrak k_h$. By the invariance of $B$ and trace, $B(X,X) = 2(n+1) \tr
(X^2)$. By polarization, \be B(X,Y) = 2(n+1) \tr (XY),  \ \ \ \
X,Y \in \mathfrak{su} (n,1). \ee Hence when $X \not =
Y\in\mathfrak B$, we have $\lp X, X \rp = 4(n+1)$ and $\lp X, Y
\rp =0$.
\end{proof}

\begin{corollary}  \label{metric}
The matrix representation for the canonical metric $g$ of
$\SU(n,1)$ is the square $n^2 + 2n$ diagonal matrix
\[\left(\begin{array}{cccc}4n+4 &   &   &   \\  & 4n+4 &   &   \\
&   & \ddots &   \\  &   &   & 4n+4\end{array}\right).\]
\end{corollary}

We will be interested in the metric on $\SU(n,1)$ that induces \textit{holomorphic} sectional curvature $-1$ on the quotient $\SU(n,1)/U(n)$. To this end, we scale the canonical metric by a factor of $\frac{1}{n+1}$. Formally,

\begin{definition}\label{mainmetric} Let $g$ be the canonical metric on $\SU(n,1)$. The metric $\tilde g$ on $\SU(n,1)$ is defined by \[\tilde g=\frac{1}{n+1}g.\]
\end{definition}

Finally, a canonical metric on a Lie algebra $\mathfrak g$ induces a norm given by

\[\norm X=\langle X, X \rangle^{1/2}.\]
Let, \[N(\ad X)=\sup\{\norm{\ad X(Y)}\,|\,Y\in\mathfrak g, \norm{Y}=1\},\]

\begin{equation}\label{Cs} C_1=\sup\{N(\ad X)\,|\,X\in\mathfrak p, \norm X=1\}  \mbox{ and }C_2=\sup\{N(\ad X)\,\|\,X\in\mathfrak k, \norm X=1\}.\end{equation}

The appendix to \cite{Wa1} includes a table of the constants $C_1$ and $C_2$ for noncompact and nonexceptional Lie groups. The values for $\SU(n,1)$ are \[C_1=C_2=(n+1)^{-1/2}.\] However, with respect to the scaled canonical metric $\tilde g$, we have  \begin{equation}\label{Cval} C_1=C_2 = 1.\end{equation}

\subsection{The Sectional Curvatures of $\SU(n,1)$}

A \textit{connection} $\nabla$ on the tangent bundle of a manifold can be expressed in terms of a left invariant metric by the \textit{Koszul formula}. For any left invariant vector fields $X,Y,Z$ we have
\[
\lp \nabla _XY , Z \rp =  \frac 12\left\{ \lp [X,Y], Z\rp   - \lp Y, [X, Z] \rp -  \lp X, [Y,Z]\rp  \right\}.  
\]
The \textit{curvature tensor} of a connection $\nabla$ is defined by
\[ 
R(X,Y)Z=\nabla_X\nabla_YZ-\nabla_Y\nabla_XZ-\nabla_{[X,Y]}Z.
\]
In \cite{Ade2}, we derived the curvature formulas for the canonical metric of a semisimple noncompact Lie group. These formulas also apply to $\tilde g$, as it is simply a scale of the canonical metric. Let $U,V, W \in \mathfrak k$ and $X,Y, Z \in \mathfrak p$ denote left invariant vector fields. 

\begin{prop}  \label{curvatures}
\ba
R(U,V)W & = & \frac 14 [ [V,U], W],   \label{R-UVW}  \\
 R(X,Y)Z & = & - \frac 74 [ [X,Y],  Z ],    \label{R-XYZ}\\
 R(U,X)Y & = & \frac 14 [[X,U], Y]  - \frac 12 [[Y,U], X] , \label{R-UXY}  \\
 R(X,Y)V & = & \frac 34 [X, [V,Y]] + \frac 34 [Y, [X,V]].  \label{R-XYV}
\ea
In particular,
\ba
\lp R(U,V)W, X \rp & = & 0,   \label{R-mix1}\\
\lp R(X,Y)Z, U\rp & = &0,  \label{R-mix2}  \\
\lp R(U,V)V,U\rp & =  & \frac 14 \|[U,  V]\|^2,  \label{R-UV}\\
\lp R(X,Y)Y,X \rp & = & -  \frac 74  \|[X,  Y]\|^2, \label{R_XY}\\
\lp R(U,X)X,U\rp & =  & \frac 14   \|[U,  X]\|^2.   \label{R-UX}
\ea
\end{prop}

 The \textit{sectional curvature} of the planes spanned by $X,Y \in \mathfrak g$ is denoted and defined by \[K(X,Y)=\frac{\lp R(X,Y)Y,X \rp}{\|X\|^2\|Y\|^2-\lp X,Y\rp^2}.\] 

\begin{prop}\label{seccurbas}
The sectional curvature of $\SU(n,1)$ with respect to the metric $\tilde g$ at the planes spanned by standard basis elements is bounded above by $1/4$.
\end{prop}
\begin{proof}
Since the basis elements are mutually orthogonal, the sectional curvature at the plane spanned by any distinct elements $X, Y \in \mathfrak B$ is given by  \[K(X, Y) =  \frac{\lp R(X,Y)Y,X \rp}{\|X\|^2\|Y\|^2} .\]
By (\ref{R-UV}), (\ref{R_XY}), (\ref{R-UX}) and Proposition~\ref{multtable}, the largest sectional curvature spanned by basis directions are the planes spanned by $h_j, i\alpha_{j,n+1}$ or $h_j, \beta_{j,n+1}$. And
\be
K(h_j, i\alpha_{j,n+1})  =   \frac{ \frac 14  \|[h_j, i\alpha_{j,n+1}]\|^2}{\|h_j\|^2 \|i\alpha_{j,n+1}\|^2}  = \frac 14 \frac{ \|-2 \beta_{j,n+1}\|^2}{4 \cdot 4} = 1/4.
\ee
\end{proof}

\begin{prop}\label{seccur} The sectional curvatures of $\SU(n,1)$ with respect to $\tilde{g}$ are bounded above by
\[
\frac 14 + 2 \cdot \frac 14  +  2 \cdot \frac 64 \cdot (2n+1)  + 2 \cdot \frac 34 \cdot (2n+1)  =  \frac{36n + 21}{4}. \]
\end{prop}

\begin{proof} Again with $U,V \in \mathfrak{k}\mbox{ and } X, Y \in \mathfrak{p}$, we have by (\ref{R-mix1}) and (\ref{R-mix2})
\begin{eqnarray*}
\langle R(X+U,Y+V)Y+V,X+U \rangle  & = & \langle R(X,Y)Y,X\rangle + \langle R(U,V)V,U\rangle + \langle R(U,Y)Y,U\rangle \\
&   +& \langle R(X,V)V,X \rangle  + 2\langle R(X,Y)V, U\rangle + 2\langle R (X,V)Y,U \rangle.
\end{eqnarray*}
Assume that $\|U+X\|=1,~ \|V+Y\| = 1 \mbox{ and }\langle U+X, V+Y \rangle = 0$.
Write \[U = \sum_{j<k} (a_{jk} \alpha_{jk} + b_{jk} i\beta_{jk} ) +\sum_j c_jh_j,~  \ \  V = \sum_{j<k} (a_{jk}' \alpha_{jk} + b_{jk}' i\beta_{jk} ) +\sum_j c_j' h_j, \]
\[ X = \sum_{j=1}^n( e_j i\alpha_{j, n+1} + f_j \beta_{j,n+1}), ~\ \  Y = \sum_{j=1}^n ( e_j' i\alpha_{j, n+1} + f_j' \beta_{j,n+1}). \]
Note that
\[
\sum_{j<k} |a_{jk}|^2 + |b_{jk}|^2 + c_j^2,\   \sum_{j<k}
|a'_{jk}|^2 + |b'_{jk}|^2 + |c'_j|^2, \ \sum_{j=1}^n  e_j^2+
f_j^2, \ \sum_{j=1}^n  |e'_j|^2 + |f_j'|^2 \le \frac 14.
\]
 By (\ref{R-UVW}) and (\ref{R-UXY}),

   \begin{gather*}
R(U,V)V = \frac 14 [[V,U], V]  = -\frac 14  \ad V \circ \ad V (U)\\
 \tag*{\text{and}}R(U,Y)Y   =  - \frac 14 [[Y,U],Y] = \frac 14 \ad Y\circ \ad Y (U).
 \end{gather*}
 Therefore, by (\ref{Cval}),
 \begin{gather*}
 \langle R(U,V)V,U\rangle \le \frac 14 C_2^2  =  \frac 14\\
\tag*{\text{and}} \langle R(U,Y)Y,U\rangle \le \frac 14 C_1^2 =  \frac 14.
 \end{gather*}

 By (\ref{R-XYV}), we have
\[
\langle R(X,Y)V, U\rangle = - \frac 34 \left( \lp [U,X], [V,Y] \rp + \lp [V,X], [U,Y] \rp \right).
\]
From (\ref{alpha-beta-n+1})--(\ref{h-alpha-n+1}),
\begin{eqnarray*}
\| [U,Y] \|^2   & = &    \| [  \sum_{j<k} (a_{jk} \alpha_{jk} + b_{jk} i\beta_{jk} ) +\sum_j c_jh_j,  \sum_{l=1}^n ( e_l' i\alpha_{l, n+1} + f_l' \beta_{l,n+1})]\|^2  \\
& = &    \|  \sum_l \left\{ \left( \sum_j (a_{lj} e_j'  + b_{lj}f_j' + c_jf_l') +c_lf_l'\right) i\alpha_{l, n+1} \right.
 \\
& & \left.  + \left( \sum_j (a_{lj} f_j' - b_{lj}e_j' - c_j e_l' ) - c_le_l'\right) \beta_{l, n+1} \right\} \|^2 \\
& = &    4  \sum_l  \left\{ \left( \sum_j (a_{lj} e_j'  + b_{lj}f_j' + c_jf_l') +c_lf_l'\right)^2  + \left( \sum_j (a_{lj} f_j' - b_{lj}e_j' - c_je_l' ) - c_le_l'\right)^2  \right\} \\
& \le & 4 \sum_l \left( 2\sum_j (a_{lj}^2 + b_{lj}^2 +c_j^2) \cdot \sum_j (|e_j'|^2 + |f_j'|^2 + |f_l'|^2) + 2 c_l^2 |f_l'|^2 \right) \\
&& + 4 \sum_l \left( 2\sum_j (a_{lj}^2 + b_{lj}^2 +c_j^2) \cdot \sum_j ( |f_j'|^2 + |e_j'|^2 + |e_l'|^2) + 2 c_l^2 |e_l'|^2 \right) \\
& \le & 8 \sum_l \left(\frac 14  \left[ \frac 14 + n |f_l'|^2 \right]  +  c_l^2|f_l'|^2 \right)  + 8\sum_l \left( \frac 14  \left[ \frac 14 + n |e_l'|^2 \right] + c_l^2|e_l'|^2 \right) \\
& \le & 2n+1.
 \end{eqnarray*}
 Here we define $a_{kj} = -a_{jk}, b_{kj} = b_{jk}, b_{jj} = 0$.
Hence \[ \langle R(X,Y)V, U\rangle \le  \frac{6}{4} \cdot (2n+1).
\]
Similarly,  by (\ref{R-UXY}),
\[
\langle R (X,V)Y,U \rangle \le \frac{3}{4} \cdot  (2n+1).
\]
\end{proof}

\sect{The Volume of Complex Hyperbolic Orbifolds}\label{vcho}

This section concludes with a proof of Theorem \ref{Bound}. We begin by assembling the required prerequisites. First, a result due to H. C. Wang is used to produce a value such that the fundamental domain of any discrete subgroup $\Gamma$ of $\SU(n,1)$ contains a metric ball of that radius. Next, a comparison theorem of Gunther is employed in order to bound from below the volume of a ball in $\SU(n,1)$. In the third subsection, a Riemannian submersion from the quotient of $\SU(n,1)$ by $\Gamma$ onto the complex hyperbolic orbifold defined by $\Gamma$ is constructed.

\subsection{H. C. Wang's Result} Let G be a semisimple Lie group without compact factor. Let $C_1$ and $C_2$ be the corresponding constants as defined in (\ref{Cs}). The number $R_G$ is defined to be the least positive zero of the real-valued function \be\label{Fval}F(t)=\exp C_1t-1+2\sin C_2t-\frac{C_1t}{\exp C_1t-1}~.\ee

The following result (Theorem 5.2 in \cite{Wa1}) gives Wang's quantitative version of the well-known result of Kazhdan-Margulis \cite{KM}.

\begin{theorem}[Wang]\label{wang2} Let $G$ be a semisimple Lie group without compact factor, let $e$ be the identity of $G$, let $\rho$ be the distance function derived from a canonical metric, and let \[B_G=\{x\in G:\rho(e,x)\leq R_G\}.\] Then for any discrete subgroup $\Gamma$ of G, there exists $g\in G$ such that $B_G\cap g \Gamma g^{-1} =\{e\}$.\end{theorem}

In addition, Wang showed that number $R_G$ is less than the injectivity radius of $G$. Consequently, the volume of the fundamental domain of any discrete subgroup $\Gamma$ of $G$, when viewed as a group of left translations of $G$, is bounded from below by the volume of a $\rho$-ball of radius $R_G/2$.

By (\ref{Cval}) and (\ref{Fval}),
\be\label{radius}R_{\SU(n,1)}\approx 0.277\dots\ee

\subsection{Gunther's Result}

Let $V(d,k,r)$ denote the volume of a ball of radius $r$ in the  complete simply connected Riemannian manifold of dimension $d$ with constant curvature $k$. A proof of the following comparison theorem can be found in \cite[Theorem 3.101]{Gall}.

\begin{theorem}[Gunther]\label{gun}Let $M$ be a complete Riemannian manifold of dimension $d$. For $m\in M$, let $B_{m}(r)$ be a ball which does not meet the cut-locus of $m$.

If the sectional curvatures of $M$ are bounded above by a constant $b$, then
\[\Vol[B_{m}(r)]\geq V(d,b,r).\]
\end{theorem}

\begin{prop}\label{numer} Let $\Gamma$ be a discrete subgroup of $\SU(n,1)$. Then \[\Vol[\SU(n,1)/\Gamma] \geq V(d_0,k_0,r_0),\] where $d_0=n^2+2n$, $k_0=\dfrac{36n+21}{4}$ and $r_0=0.1385$.
\end{prop}
\begin{proof} The inequality is immediate from Theorems~\ref{wang2} and \ref{gun}. The values of $d_0$, $k_0$ and $r_0$ follow from Definition~\ref{basiselt}, Proposition~\ref{seccur} and (\ref{radius}), respectively.\end{proof}

\subsection{Riemannian Submersions}

Let $(M,g)$ and $(N,h)$ be Riemannian manifolds and $q:M\rightarrow N$ a surjective submersion. The map $q$ is said to be a \textit{Riemannian submersion} if \[g(X,Y)=h(dq X,dq Y)\text{\,\,\,whenever\,\,\,}X,Y\in(\Ker dq)^{\bot}_{x}\text{\,\,\,for some\,\,\,} x\in M.\]

The following elementary results are proved in \cite{Ade2}.

\begin{lemma}\label{totgeo}
Let $G$ be a semisimple Lie group and let $\mathfrak g$ be its Lie algebra, with Cartan decomposition $\mathfrak g=\mathfrak k\oplus\mathfrak p$. Let $K$ be the maximal compact subgroup of $G$ with Lie algebra $\mathfrak k$. Then, with respect to the canonical metric, $K$ is totally geodesic in $G$.
\end{lemma}

\begin{lemma}\label{resub}
Let $K \ra M \stackrel{q}{\ra} N$ denote a fiber bundle where $q$ is a Riemannian submersion and  $K$ is a compact and totally geodesic submanifold  of $M$. Then for any subset $Z \subset N$, \[\vol [q^{-1}(Z)] =  \vol[Z] \cdot \vol[K]\]
\end{lemma}

Let $X, Y$ be orthonormal vector fields on $N$ and let $\tilde X,\tilde Y$ be their horizontal lifts to $M$. O'Neill's formula (see e.g. \cite[Page 127]{Gall}), relates the sectional curvature of the base space of a Riemannian submersion with that of the total space:
\be\label{ON} K_b(X, Y)=K_t(X, Y)+\frac 3 4\|[X, Y]^{\perp}\|^2,\ee where $Z^{\perp}$ represents the vertical component of $Z$.

Recall the definitions and notation of Section~\ref{sun} and consider the quotient map \[\pi : \SU(n,1) \rightarrow \SU(n,1)/U(n).\] The restriction of the inner product $\langle X,Y\rangle$, defined on $\mathfrak{su}(n,1) = \mathfrak{k} \oplus \mathfrak{p}$, to $d_e\pi(\mathfrak{p})=T_{\pi(e)}\SU(n,1)/U(n)$, induces a Riemannian metric on the quotient space. With respect to these metrics, the map $\pi$ is a Riemannian submersion.

We now show that if $\SU(n,1)/U(n)$ is equipped with the restriction of the scaled canonical metric $\tilde g$, it has constant \textit{holomorphic sectional curvature} $-1$. It then follows that $\pi$ is a Riemannian submersion from $\SU(n,1)$ to $\mathbf H^n_{\mathbb C}$, complex hyperbolic $n$-space.

Let $X\in \mathfrak{p} $ represent both a unit vector field on $\SU(n,1)/U(n)$ as well as its horizontal lift. Write  $X = \sum_{j=1}^n( a_j i\alpha_{j, n+1} + b_j \beta_{j,n+1})$. Since $\|X\|=1$, we have $\sum_j (a_j^2+b_j^2)=\frac 14$.

From the identification of complex structure
\[ \left( \begin{array}{cc} 0 & \xi^*  \\ \xi & 0 \end{array}
\right) \  \longleftrightarrow \  \xi \  \ \ \ \ \mbox{for} \ \xi
\in \mathbb C,
\]
$JX =\sum_{k=1}^n( -b_k i\alpha_{k, n+1} + a_k
\beta_{k,n+1})$.
 By (\ref{R_XY}), the
holomorphic sectional curvature 
\[ K_t(X, JX) = \lp R(X, JX)JX, X \rp = -\frac 74 \| [X, JX]
\|^2.\] By (\ref{sigref})-(\ref{n+1-n+1}),
\ban
[X, JX]^{\perp} &  = &  [X, JX]  \\
& = &   \sum_{j,k}(a_kb_j -a_jb_k) \alpha_{jk} +
\sum_{j\not=k} (a_ja_k +b_jb_k)i\beta_{jk} + 2\sum_j
(a_j^2+b_j^2)h_j \\
& = & 2 \left\{ \sum_{j <k} \left[ (a_kb_j -a_jb_k) \alpha_{jk} +  (a_ja_k +b_jb_k)i\beta_{jk} \right] + \sum_j
(a_j^2+b_j^2)h_j \right\} .
\ean
Hence \ban K_b(X, JX) & = & -\| [X, JX]\|^2 \\
&  = &  -4 \cdot 4 \left\{ \sum_{j<k}  \left[ (a_kb_j -a_jb_k)^2 + (a_ja_k
+b_jb_k)^2  \right]+  (\sum_j (a_j^2 + b_j^2)^2 )\right\} \\
&  = & -4 \cdot 4 \left[  \sum_j (a_j^2+b_j^2) \sum_k
(a_k^2+b_k^2)\right]  = -1. \ean

For a discrete group $\Gamma<\SU(n,1)$ and complex hyperbolic orbifold $Q=\Gamma\backslash\mathbf H^n_{\mathbb C}$, the map $\pi$ induces another Riemannian submersion \[\pi^{\prime}: \SU(n,1)/\Gamma\rightarrow Q.\] The fibers of $\pi^{\prime}$ on the smooth points of $Q$ are totally geodesic embedded copies of $U(n)$.

\subsection{Main Result} We now give a proof of Theorem \ref{Bound}, which for convenience is restated below.

{\bf Theorem \ref{Bound}}\quad{\sl The volume of a complex hyperbolic $n$-orbifold is bounded below by $\mathcal C(n)$, an explicit constant depending only on dimension, given by

\[\mathcal C(n)=\frac{2^{n^2+n+1} \pi^{n/2}(n-1)!\,(n-2)!\cdots!\, 3! \,2! \,1!}{(36n+21)^{(n^2+2n)/2}\Gamma((n^2+2n)/2)}\int_{0}^{\min[0.06925\sqrt{36n+21},\pi]}\sin^{n^2+2n-1}\rho\,\,\,d\rho.\]}

\begin{proof} Let $Q$ be a complex hyperbolic $n$-orbifold. By the last paragraph of the previous subsection, Proposition~\ref{numer} and Lemma~\ref{resub},
  \[ V(d_0,k_0,r_0)\leq \Vol[\SU(n,1)/\Gamma]\leq \vol [\pi^{-1}(Q)] =\Vol[Q]\cdot\Vol[U(n)]\label{frac}.\] The proof follows from the following two observations:

  The volumes of the classical compact groups are given explicitly in \cite[Page 399]{Gilmore}. For the unitary group, the volume with respect to the metric $\tilde g$ is
 \[\Vol[U(n)]=\frac{2^n\pi^{(n^2+n)/2}}{(n-1)!\,(n-2)!\cdots!\, 3! \,2! \,1!}\label{UN}\,\,.\]

The complete simply connected Riemannian manifold with constant curvature $k>0$ is the sphere of radius $k^{-1/2}$. By explicit computation we have \[V(d,k,r)= \frac{2(\pi/k)^{d/2}}{\Gamma(d/2)}\int_{0}^{\min\left[rk^{1/2},\pi\right]}\sin^{d-1}\rho\,\,\,d\rho.\]
\end{proof}

\sect{Volume Bounds}\label{vb}

In this section, we give an outline of current results on complex hyperbolic volume. The isometries of complex hyperbolic space are classified into three types: elliptic, parabolic and loxodromic (see e.g. \cite{Gold}). A finite volume complex hyperbolic orbifold $\mathbf H^n_{\mathbb C}/ \Gamma$ is : a \textit{manifold} when $\Gamma$ does not contain elliptic elements; \textit{closed} (or compact) when $\Gamma$ does not contain parabolic elements and \textit{cusped} (or noncompact) when it does;  and \textit{arithmetic} when $\Gamma$ can be derived by a specific number-theoretic construction (see e.g. \cite{Sto2}). 

\subsection{Complex Hyperbolic Manifolds}

In \cite{HerPau}, Hersonsky and Paulin used the Chern-Gauss-Bonnet formula to prove that the smallest volume of a closed complex hyperbolic 2-manifold is $8\pi^2$. The work of Prasad and Yeung \cite{PraY}, \cite{PraY2} and Cartwright and Steger \cite{Cart} on the classification of fake projective planes has produced 51 explicit examples. An article by Yeung \cite{Yeung} shows that this list is exhaustive. Xie, Wang and Jiang \cite{XWJ} give, for each dimension $n$, an explicit lower bound for the largest number such that every complex hyperbolic $n$-manifold contains an embedded ball of that radius. 

\subsection{Cusped Complex Hyperbolic Manifolds}

Parker \cite{Parker} proved that the smallest volume of a cusped (and so of any) complex hyperbolic 2-manifold is $8\pi^2/3$ and found one such example. A total of eight such manifolds are given by Stover \cite{Sto}. Volume bounds for noncompact complex hyperbolic manifolds in terms of dimension and the number of cusps were given by Hersonsky and Paulin \cite{HerPau} and Parker \cite{Parker}. These bounds were later improved by Hwang \cite{Hwang} using methods from algebraic geometry. Kim and Kim \cite{Kims} give a bound that is sharper than \cite{Hwang} in the case where a complex hyperbolic manifold has exactly one cusp.

\subsection{Complex Hyperbolic Orbifolds} Parker \cite{Parker} also proved that the volume of a cusped complex hyperbolic 2-orbifold is bounded below by $1/4$. He identified two orbifolds with volume $\pi^2/27$, and conjectured them to be the cusped complex hyperbolic 2-orbifolds of minimum volume. Extending a result for the real hyperbolic case \cite{Ade1}, Fu, Li and Wang \cite{FLW} obtained a lower bound for the volume of a complex hyperbolic orbifold, depending on dimension and the maximal order of torsion in the orbifold fundamental group.

\subsection{Arithmetic Complex Hyperbolic Orbifolds} The smallest known complex hyperbolic 2-manifolds, closed or cusped, are arithmetically defined. Stover \cite{Sto} proved that the orbifolds considered by Parker \cite{Parker} are the smallest volume cusped arithmetic complex hyperbolic 2-orbifolds. In \cite{Sto2}, Emery and Stover determine, for each dimension $n$, the smallest volume cusped arithmetic complex hyperbolic orbifold. It is shown that, as $n$ varies, minimum volume among all cusped arithmetic complex hyperbolic orbifolds is realized in dimension 9. Smaller volume orbifolds have been found in the  compact case. For example, Sauter \cite{Saut} exhibited a closed arithmetic complex hyperbolic 2-orbifold with volume $\pi^2/108$.

%%%%%%%%%%%%%%%%%%%%%%%%%%%%%%%%%%%%%%%%%%%%%%%%%%%%%%%%%%%%%%%%%

 \section*{acknowledgments}

 The authors are grateful to Dick Canary, Ben McReynolds and Matthew Stover for useful conversations.

\bibliography{ProjectThree.bib}

\end{document}